\documentclass[11pt,reqno]{amsart}

\usepackage[a4paper,margin=31mm]{geometry}
\usepackage[T1]{fontenc}
\usepackage[utf8]{inputenc}
\usepackage{lmodern}
\usepackage{microtype}
\usepackage{mathtools,amssymb,mathrsfs}
\usepackage{booktabs,array}
\usepackage[inline]{enumitem}
\usepackage[dvipsnames]{xcolor}
\usepackage[colorlinks=true,
            linkcolor=MidnightBlue,
            citecolor=BrickRed,
            urlcolor=RoyalBlue]{hyperref}
\usepackage[nameinlink,noabbrev]{cleveref}

\allowdisplaybreaks
\setlist[itemize]{leftmargin=2em,itemsep=2pt,topsep=4pt}
\setlist[enumerate]{leftmargin=2.3em,itemsep=2pt,topsep=4pt}

\newtheorem{theorem}{Theorem}[section]
\newtheorem{proposition}[theorem]{Proposition}
\newtheorem{lemma}[theorem]{Lemma}
\newtheorem{corollary}[theorem]{Corollary}
\newtheorem{remark}[theorem]{Remark}
\newtheorem{example}[theorem]{Example}
\theoremstyle{definition}

\newcommand{\R}{\mathbb R}
\newcommand{\SX}{S_X}

\newcommand{\M}{\mathcal M}
\newcommand{\Ssk}{\mathcal S}
\newcommand{\Jsk}{\mathcal J}

\newcommand{\cTau}{c_{\tau}}
\newcommand{\rTau}{r_{\tau}}
\newcommand{\hTau}{h_{t}(\tau)}
\newcommand{\norm}[1]{\left\lVert #1\right\rVert}
\newcommand{\abs}[1]{\left\lvert #1\right\rvert}
\newcommand{\ip}[2]{\left\langle #1,#2\right\rangle}

\title[]{}

\author{}
\thanks{Affiliation: School of Mathematics and Statistics, Anqing Normal University, Anqing 246133, P. R. China.}

\hypersetup{
  pdftitle={Skew Schaffer Type Constants},
  pdfauthor={Ayla Lee},
  pdfsubject={Geometric constants in Banach spaces},
  pdfkeywords={Schaffer constant, skew James constant, uniform nonsquareness, normal structure}
}

\begin{document}

\begin{center}
	{\LARGE\bfseries Skew Sch\"affer Type Constants}\par
	\vspace{1.1em}
	{\large Ying Xu$^{1}$ and Qi Liu$^{1*}$}\par
	\vspace{0.8em}
	$^{1*}$School of Mathematics and Statistics, Anqing Normal University, Anqing, 246133,\par
	Anhui, China.\par
	\vspace{1.0em}
	\textit{Corresponding author(s).} E-mail(s): \href{mailto:liuq67@aqnu.edu.cn}{liuq67@aqnu.edu.cn};\par
	\textit{Contributing authors:} \href{mailto:Y26220046@stu.aqnu.edu.cn}{Y26220046@stu.aqnu.edu.cn};\par
\end{center}

\begin{abstract}
We introduce the skew Sch\"affer type constant as an infimal counterpart of the recently developed skew James type constant \cite{RaoWuLiuNiWang} and establish its basic geometric theory in Banach spaces. The new constant admits sharp universal bounds and characterizes uniform nonsquareness through a strict gap condition. Exact values are obtained in Hilbert and Lebesgue spaces, where an exact relation with the skew James type constant is established. We further derive a James-modulus estimate leading to sufficient conditions for uniform normal structure.
\end{abstract}

\maketitle

\section{Introduction}

Geometric constants compress metric information about the unit sphere of a
Banach space into quantities that can be compared, computed, and used in
fixed point theory.  Their classical point of departure is the
Jordan--von Neumann identity
\[
  \norm{x+y}^{2}+\norm{x-y}^{2}
  =2\norm{x}^{2}+2\norm{y}^{2},
\]
which characterizes inner-product norms \cite{JordanVonNeumann,Day}.  The
James and Sch\"affer constants subsequently provided two complementary ways
to measure square configurations in the unit sphere; see
\cite{James,Schaffer,KatoMaligrandaTakahashi}.  Their usefulness is not merely
descriptive: strict inequalities for these constants imply reflexivity,
normal structure, or fixed point properties for nonexpansive mappings.

Power-mean versions of the James and Sch\"affer constants were organized by
Takahashi \cite{Takahashi} and studied further in
\cite{WangYang,ZuoTang,MartiniPapiniWu}.  In particular, Yang and Wang
\cite{WangYang} investigated fundamental properties of the James type
constant and established several relations with other geometric quantities
in Banach spaces.  Their work provided important tools for describing
quantitative nonsquareness through supremal geometric invariants and
motivated subsequent developments of parameterized and skew versions of
James type constants.
If
\[
  \M_t(a,b)=\left(\frac{a^t+b^t}{2}\right)^{1/t}
  \quad (1\leq t<\infty),
  \qquad
  \M_\infty(a,b)=\max\{a,b\},
\]
then the usual Sch\"affer type family uses the symmetric pair
$x+\tau y,x-\tau y$.  This symmetry is powerful: after independent
normalization of two vectors, the sum of the two $t$-powers remains an even
convex function of each scaling variable.  It is precisely this observation
that underlies several homogeneous representations and estimates in the
classical theory.

Another line of work starts from the Euler--Lagrange transform
\[
  (x,y)\longmapsto (x+\tau y,\,\tau x-y).
\]
Quadratic supremal invariants associated with this transform were considered
in \cite{LiuLi,LiuZhouSarfrazLi,NiLiuZhou}.  More recently, Rao, Wu, Liu, Ni,
and Wang introduced the full power-mean supremal invariant associated with the
skew transform \cite{RaoWuLiuNiWang}:
\begin{equation}\label{eq:skew-james-intro}
  \Jsk_t[\tau,X]
  =\sup_{x,y\in \SX}
     \M_t\bigl(\norm{x+\tau y},\norm{\tau x-y}\bigr),
\end{equation}
which is called the skew James type constant.  In contrast, the corresponding
infimal geometry has not been developed.  The purpose of the present paper is
to fill this gap by introducing and studying the skew Sch\"affer type
constant.

Let $X$ be a nonzero real Banach space, let $\tau\geq0$, and let
$1\leq t\leq\infty$.  Our central object is
\begin{equation}\label{eq:def-S-intro}
  \Ssk_{X,t}(\tau)
  :=\inf_{x,y\in \SX}
     \M_t\bigl(\norm{x+\tau y},\norm{\tau x-y}\bigr).
\end{equation}
We call \eqref{eq:def-S-intro} the \emph{skew Sch\"affer type constant}.  At
$\tau=1$ it reduces to the familiar power-mean Sch\"affer invariant, but for
$\tau\neq1$ it differs essentially from the pair
$x+\tau y,x-\tau y$.  Thus the adjective \emph{skew} records an actual
change of geometry rather than a change of notation.

The mechanism behind our results is the matrix
\begin{equation}\label{eq:A-intro}
  A_\tau=
  \begin{pmatrix}1&\tau\\[2pt]\tau&-1\end{pmatrix},
  \qquad
  A_\tau^2=(1+\tau^2)I_2.
\end{equation}
The scaled-involution identity in \eqref{eq:A-intro} has three consequences.
First, it produces a sharp lower bound whose equality case detects
asymptotic copies of the square.  Second, on $L_p$ it turns the new constant
into the minimum modulus of $A_\tau$ on $\ell_p^2$.  Third, although a single
skew pair is not separately even under rescaling, adding the pair with the
variables reversed restores the required convex symmetry.  This last
device yields a complete analogue of the James-modulus estimate used in
\cite{ZuoTang}, now with a proof that makes the lost symmetry explicit.

The main conclusions are as follows.
\begin{itemize}
  \item The sharp chain
  \[
     \frac{1+\tau^2}{1+\tau}
     \leq \Ssk_{X,t}(\tau)
     \leq \M_t(1+\tau,\abs{1-\tau})
     \leq \Jsk_t[\tau,X]
     \leq 1+\tau
  \]
  holds for every nonzero real Banach space.
  \item Except for the degenerate pair $(t,\tau)=(1,1)$, strictness in the
  first inequality is equivalent to uniform nonsquareness.
  \item If $H$ is a Hilbert space of dimension at least two, both the
  infimal and supremal skew constants are computed exactly.  If
  $X=L_p(\mu)$ has dimension at least two, then
  \[
     \Ssk_{X,p}(\tau)
       =\frac{1+\tau^2}{\norm{A_\tau}_{\ell_p^2\to\ell_p^2}},
     \qquad
     \Jsk_p[\tau,X]=\norm{A_\tau}_{\ell_p^2\to\ell_p^2}.
  \]
  \item For $0\leq\tau,\varepsilon\leq1$ and $1\leq t<\infty$,
  \[
  \Ssk_{X,t}(\tau)J(\varepsilon,X)
  \leq
  \M_t\bigl(1+\tau+\varepsilon(1-\tau),
             1-\tau+\varepsilon(1+\tau)\bigr),
  \]
  where $J(\varepsilon,X)$ is the classical James modulus.  This estimate
  yields a flexible two-parameter criterion for uniform normal structure.
\end{itemize}

The paper is organized accordingly.  Parameter structure is developed in
\cref{sec:basic}, and exact model computations are given in
\cref{sec:models}.  Sharp bounds and the lower-bound gap are analyzed in
\cref{sec:unsquare}, while the James-modulus and normal-structure
consequences occupy \cref{sec:normal}.

\section{Parameter structure}\label{sec:basic}

Throughout, $X$ is a nonzero real Banach space, $\SX$ is its unit sphere,
and $1\leq t\leq\infty$.  We use
\[
  \cTau:=1+\tau^2,
  \qquad
  \rTau:=\frac{1+\tau^2}{1+\tau},
  \qquad
  \hTau:=\M_t(1+\tau,\abs{1-\tau}).
\]
In addition to \eqref{eq:def-S-intro}, we use the supremal constant
$\Jsk_t[\tau,X]$ from \eqref{eq:skew-james-intro}.  The two quantities are
the lower and upper radii, respectively, of the same power-mean profile on
$\SX\times\SX$.

We first record elementary facts about two-point power means.  If
$1\leq r<s\leq\infty$, then
\begin{equation}\label{eq:mean-comparison}
  \M_r(a,b)\leq\M_s(a,b)
  \leq 2^{1/r-1/s}\M_r(a,b),
\end{equation}
where $1/\infty=0$.  Moreover, as the normalized $\ell_t^2$ norm,
\begin{equation}\label{eq:mean-lipschitz}
 \abs{\M_t(a_1,a_2)-\M_t(b_1,b_2)}
 \leq \M_t(\abs{a_1-b_1},\abs{a_2-b_2}).
\end{equation}

\begin{proposition}\label{prop:structure}
For every nonzero Banach space $X$, the following statements hold.
\begin{enumerate}[label=\textup{(\roman*)}]
  \item $\Ssk_{X,t}(0)=\Jsk_t[0,X]=1$.
  \item For $\tau>0$,
  \[
     \Ssk_{X,t}(\tau)=\tau\Ssk_{X,t}(\tau^{-1}),
     \qquad
     \Jsk_t[\tau,X]=\tau\Jsk_t[\tau^{-1},X].
  \]
  \item If $1\leq r<s\leq\infty$, then
  \begin{align*}
     \Ssk_{X,r}(\tau)&\leq\Ssk_{X,s}(\tau)
       \leq2^{1/r-1/s}\Ssk_{X,r}(\tau),\\
     \Jsk_r[\tau,X]&\leq\Jsk_s[\tau,X]
       \leq2^{1/r-1/s}\Jsk_r[\tau,X].
  \end{align*}
  \item The functions $\tau\mapsto\Ssk_{X,t}(\tau)$ and
  $\tau\mapsto\Jsk_t[\tau,X]$ are $1$-Lipschitz on $[0,\infty)$.
\end{enumerate}
\end{proposition}

\begin{proof}
Part (i) follows directly from $\norm{x}=\norm{y}=1$.  For (ii), multiply
the expression with parameter $\tau^{-1}$ by $\tau$ and then interchange
$x$ and $y$:
\begin{align*}
 \tau\M_t\bigl(\norm{x+\tau^{-1}y},\norm{\tau^{-1}x-y}\bigr)
  &=\M_t\bigl(\norm{\tau x+y},\norm{x-\tau y}\bigr)\\
  &=\M_t\bigl(\norm{y+\tau x},\norm{\tau y-x}\bigr).
\end{align*}
Taking the infimum or supremum gives the assertion.  Part (iii) is
\eqref{eq:mean-comparison} followed by the same extremal operation.

For (iv), fix $x,y\in\SX$ and $\tau,\sigma\geq0$.  The reverse triangle
inequality gives
\[
 \abs{\norm{x+\tau y}-\norm{x+\sigma y}}\leq\abs{\tau-\sigma},
 \qquad
 \abs{\norm{\tau x-y}-\norm{\sigma x-y}}\leq\abs{\tau-\sigma}.
\]
By \eqref{eq:mean-lipschitz}, the corresponding power-mean profiles differ
by at most $\abs{\tau-\sigma}$, uniformly in $(x,y)$.  Infima and suprema of
uniformly Lipschitz families preserve the same Lipschitz constant.
\end{proof}

\section{Exact formulas in Hilbert and \texorpdfstring{$L_p$}{Lp} spaces}\label{sec:models}

The transform $A_\tau$ behaves like a rotation followed by the dilation
$\sqrt{1+\tau^2}$ in Euclidean geometry.  This gives a complete Hilbert
space calculation and serves as a benchmark for the non-Hilbert formulas.

\begin{theorem}\label{thm:hilbert}
Let $H$ be a real Hilbert space with $\dim H\geq2$.  For every $\tau\geq0$
and $1\leq t\leq\infty$,
\begin{align}
 \Ssk_{H,t}(\tau)
   &=\min\left\{\sqrt{1+\tau^2},
           \M_t(1+\tau,\abs{1-\tau})\right\},\label{eq:Hilbert-S}\\
 \Jsk_t[\tau,H]
   &=\max\left\{\sqrt{1+\tau^2},
           \M_t(1+\tau,\abs{1-\tau})\right\}.
           \label{eq:Hilbert-J}
\end{align}
Equivalently,
\[
\Ssk_{H,t}(\tau)=
\begin{cases}
 \M_t(1+\tau,\abs{1-\tau}),&1\leq t\leq2,\\[2pt]
 \sqrt{1+\tau^2},&2\leq t\leq\infty,
\end{cases}
\]
and the two cases are reversed for $\Jsk_t[\tau,H]$.
\end{theorem}

\begin{proof}
For $x,y\in S_H$, set $s=\ip{x}{y}\in[-1,1]$.  Then
\begin{align*}
 \norm{x+\tau y}^2&=1+\tau^2+2\tau s,\\
 \norm{\tau x-y}^2&=1+\tau^2-2\tau s.
\end{align*}
Every $s\in[-1,1]$ occurs because $\dim H\geq2$.  If $1\leq t<\infty$,
the $t$th power of the profile is therefore
\begin{equation}\label{eq:hilbert-profile}
 \frac{(\cTau+2\tau s)^{t/2}+(\cTau-2\tau s)^{t/2}}{2}.
\end{equation}
For $1\leq t<2$, the two summands in \eqref{eq:hilbert-profile} are
concave as functions of their nonnegative arguments; the even profile is
maximal at $s=0$ and minimal at $s=\pm1$.  For $t>2$, convexity reverses
the two extrema.  At $t=2$ the profile is identically $\cTau$.  The
endpoint $t=\infty$ follows directly by minimizing or maximizing
\[
 \max\{(\cTau+2\tau s)^{1/2},
        (\cTau-2\tau s)^{1/2}\}.
\]
This proves the piecewise formulas.  The min--max forms
\eqref{eq:Hilbert-S}--\eqref{eq:Hilbert-J} also follow from the power-mean
comparison with the quadratic mean.
\end{proof}

We next obtain an exact formula on Lebesgue spaces.  For
$1\leq p\leq\infty$, let
\[
 \nu_p(\tau):=\norm{A_\tau}_{\ell_p^2\to\ell_p^2},
 \qquad
 m_p(\tau):=\inf_{\norm{z}_p=1}\norm{A_\tau z}_p.
\]
Since $A_\tau^{-1}=\cTau^{-1}A_\tau$,
\begin{equation}\label{eq:minmod-involution}
  m_p(\tau)=\frac{1}{\norm{A_\tau^{-1}}_{p\to p}}
  =\frac{\cTau}{\nu_p(\tau)}.
\end{equation}

\begin{theorem}\label{thm:Lp}
Let $1\leq p\leq\infty$, and let $X=L_p(\Omega,\mu)$ have dimension at
least two.  Then
\begin{equation}\label{eq:Lp-main}
 \
 \Ssk_{X,p}(\tau)=m_p(\tau)=\frac{1+\tau^2}{\nu_p(\tau)},
 \qquad
 \Jsk_p[\tau,X]=\nu_p(\tau).
\end{equation}
In particular,
\begin{equation}\label{eq:product-identity}
  \Ssk_{X,p}(\tau)\,\Jsk_p[\tau,X]=1+\tau^2.
\end{equation}
Here the convention for $p=\infty$ is that the outer mean is the maximum.
\end{theorem}

\begin{proof}
Assume first that $1\leq p<\infty$.  The scalar minimum-modulus inequality
gives, almost everywhere,
\[
 |x+\tau y|^p+|\tau x-y|^p
 \geq m_p(\tau)^p\bigl(|x|^p+|y|^p\bigr).
\]
After integration, every $x,y\in S_X$ satisfy
\[
 \frac{\norm{x+\tau y}_p^p+\norm{\tau x-y}_p^p}{2}
 \geq m_p(\tau)^p.
\]
The operator-norm inequality gives the analogous upper estimate with
$\nu_p(\tau)^p$.  Hence
\begin{equation}\label{eq:Lp-prebounds}
 \Ssk_{X,p}(\tau)\geq m_p(\tau),
 \qquad
 \Jsk_p[\tau,X]\leq\nu_p(\tau).
\end{equation}

Because $\dim L_p(\mu)\geq2$, there exist disjointly supported unit vectors
$f,g\in L_p(\mu)$.  Let $z=(a,b)\in\ell_p^2$ be a unit vector.  Define
\begin{equation}\label{eq:symmetrizing-vectors}
  x=af+bg,
  \qquad
  y=bf-ag.
\end{equation}
Then $\norm{x}_p=\norm{y}_p=1$, and disjointness gives
\begin{align*}
 \norm{x+\tau y}_p^p
 &=|a+\tau b|^p+|b-\tau a|^p,\\
 \norm{\tau x-y}_p^p
 &=|\tau a-b|^p+|\tau b+a|^p.
\end{align*}
The two right-hand sides are equal to $\norm{A_\tau z}_p^p$.  Choosing $z$
to minimize and then to maximize the scalar ratio shows that both estimates
in \eqref{eq:Lp-prebounds} are attained.  Formula
\eqref{eq:minmod-involution} yields \eqref{eq:Lp-main} and
\eqref{eq:product-identity}.

For $p=\infty$, apply the pointwise inequality
\[
 \max\{|x+\tau y|,|\tau x-y|\}
 \geq m_\infty(\tau)\max\{|x|,|y|\}
\]
and its operator-norm counterpart.  Taking essential suprema proves the
analogues of \eqref{eq:Lp-prebounds}.  The same disjoint-support
construction \eqref{eq:symmetrizing-vectors}, now with
$\max\{|a|,|b|\}=1$, attains the two scalar extrema.
\end{proof}

\begin{remark}\label{rem:balance}
The operator minimum modulus is taken over arbitrary scalar pairs $(a,b)$,
whereas \eqref{eq:def-S-intro} requires two vectors of equal norm.  The
second coordinate pair $(b,-a)$ in \eqref{eq:symmetrizing-vectors} is not
cosmetic: it interchanges, up to signs, the two coordinates of
$A_\tau(a,b)$.  It therefore balances the input norms without changing the
output ratio.  This is the step that upgrades an operator estimate to an
exact sphere formula.
\end{remark}

For $1\leq p<\infty$, the scalar norm in \cref{thm:Lp} is a one-variable
optimization:
\begin{equation}\label{eq:one-variable}
 \nu_p(\tau)^p
 =\max_{s\in\R\cup\{\infty\}}
   \frac{|1+\tau s|^p+|\tau-s|^p}{1+|s|^p},
\end{equation}
where the value at infinity is understood as the corresponding limit.
Thus \eqref{eq:Lp-main} is an exact, computationally effective formula even
when no simpler elementary expression for $\nu_p(\tau)$ is available.

\begin{corollary}
\label{cor:distinguished}
Let $X=L_p(\mu)$ have dimension at least two.
\begin{enumerate}[label=\textup{(\roman*)}]
 \item If $p=1$ or $p=\infty$, then
 \[
   \Ssk_{X,p}(\tau)=\frac{1+\tau^2}{1+\tau},
   \qquad
   \Jsk_p[\tau,X]=1+\tau.
 \]
 \item If $p=2$, then
 \[
   \Ssk_{X,2}(\tau)=\Jsk_2[\tau,X]=\sqrt{1+\tau^2}.
 \]
 \item If $\tau=1$, then
 \begin{align*}
   \Jsk_p[1,X]&=2^{\max\{1/p,\,1-1/p\}},\\
   \Ssk_{X,p}(1)&=2^{\min\{1/p,\,1-1/p\}}.
 \end{align*}
 \item If $p'=p/(p-1)$, with the usual endpoint convention, then
 \[
   \Ssk_{L_p,p}(\tau)=\Ssk_{L_{p'},p'}(\tau),
   \qquad
   \Jsk_p[\tau,L_p]=\Jsk_{p'}[\tau,L_{p'}].
 \]
\end{enumerate}
\end{corollary}

\begin{proof}
The $\ell_1^2$ and $\ell_\infty^2$ matrix norms are the maximum column and
row sums, respectively, both equal to $1+\tau$.  The Euclidean identity
$A_\tau^2=\cTau I_2$ gives $\nu_2(\tau)=\sqrt{\cTau}$.  When $\tau=1$,
interpolation between $p=1$ and $p=2$, together with the test vector
$(1,0)$, gives $\nu_p(1)=2^{1/p}$ for $1\leq p\leq2$.  Duality gives
$\nu_p(1)=2^{1-1/p}$ for $2\leq p\leq\infty$.  Finally,
$A_\tau$ is symmetric, so its $\ell_p^2$ norm equals its
$\ell_{p'}^2$ norm.  Apply \cref{thm:Lp}.
\end{proof}

{
\begin{example}\label{ex:Banas-Fraczek}
Let $\lambda>1$, and consider the classical Bana\'s--Fr\k{a}czek space
\[
 X_\lambda=(\mathbb R^2,\norm{\cdot}_\lambda),
 \qquad
 \norm{(u,v)}_\lambda
 =\max\left\{\lambda |u|,\sqrt{u^2+v^2}\right\};
\]
see \cite{YangYangBF}.  Then, for the genuinely skew range
\[
 \sqrt{1-\lambda^{-2}}\leq \tau\leq1,
\]
one has the exact formula
\begin{equation}\label{eq:Banas-Fraczek-S}
 \
 \Ssk_{X_\lambda,2}(\tau)
 =\left(\frac{(1+\tau^2)(1+\lambda^{-2})}{2}\right)^{1/2}.
\end{equation}

Indeed, write $x=(a,b)$ and $y=(c,d)$ with
$\norm{x}_\lambda=\norm{y}_\lambda=1$.  By the definition of the norm,
\[
 \max\{\lambda^2a^2,a^2+b^2\}=1,
 \qquad
 \max\{\lambda^2c^2,c^2+d^2\}=1.
\]
Moreover,
\begin{align*}
 &\norm{x+\tau y}_\lambda^2+\norm{\tau x-y}_\lambda^2\\
 &\quad\geq
 \max\Bigl\{
   \lambda^2\bigl((a+\tau c)^2+(\tau a-c)^2\bigr),
   \norm{x+\tau y}_2^2+\norm{\tau x-y}_2^2
 \Bigr\}\\
 &\quad=
 (1+\tau^2)
 \max\Bigl\{
   \lambda^2(a^2+c^2),
   \norm{x}_2^2+\norm{y}_2^2
 \Bigr\}.
\end{align*}
The unit-sphere conditions imply
\[
 \max\Bigl\{
   \lambda^2(a^2+c^2),
   \norm{x}_2^2+\norm{y}_2^2
 \Bigr\}
 \geq1+\lambda^{-2}.
\]
Hence
\[
 \Ssk_{X_\lambda,2}(\tau)
 \geq
 \left(\frac{(1+\tau^2)(1+\lambda^{-2})}{2}\right)^{1/2}.
\]

For the reverse inequality, take
\[
 x=\left(\frac1\lambda,0\right),
 \qquad y=(0,1).
\]
Then $x,y\in S_{X_\lambda}$.  Since
$\tau^2\geq1-\lambda^{-2}$ and $0\leq\tau\leq1$,
\[
 \norm{x+\tau y}_\lambda^2=\lambda^{-2}+\tau^2,
 \qquad
 \norm{\tau x-y}_\lambda^2=1+\tau^2\lambda^{-2}.
\]
Therefore
\[
 \frac{\norm{x+\tau y}_\lambda^2+\norm{\tau x-y}_\lambda^2}{2}
 =\frac{(1+\tau^2)(1+\lambda^{-2})}{2},
\]
which proves \eqref{eq:Banas-Fraczek-S}.

The threshold $\tau=\sqrt{1-\lambda^{-2}}$ is intrinsic to this example:
for the extremizing pair above, it is precisely the point at which the
Euclidean component of $\norm{x+\tau y}_\lambda$ becomes dominant.  Thus
\eqref{eq:Banas-Fraczek-S} exhibits a genuine interaction between the skew
parameter and the non-Euclidean geometry of the Bana\'s--Fr\k{a}czek unit ball.
\end{example}
}
Complex interpolation \cite{BerghLofstrom} also gives a useful closed
estimate away from $\tau=1$.  Put
\[
  \delta_p:=\abs{\frac1p-\frac12}.
\]
Interpolation between $(p,\nu_p)=(1,1+\tau)$ and
$(2,\sqrt{\cTau})$, followed by duality, yields
\begin{equation}\label{eq:interpolation-norm}
 \nu_p(\tau)
 \leq (1+\tau)^{2\delta_p}\cTau^{1/2-\delta_p}.
\end{equation}
Combining this with \cref{thm:Lp} gives the following quantitative estimate.

\begin{corollary}\label{cor:interpolation}
For $1\leq p\leq\infty$ and $\dim L_p(\mu)\geq2$,
\begin{align}
 \Ssk_{L_p,p}(\tau)
 &\geq
 \frac{\cTau^{1/2+\delta_p}}{(1+\tau)^{2\delta_p}},
 \label{eq:S-interpolation}\\
 \Jsk_p[\tau,L_p]
 &\leq
 (1+\tau)^{2\delta_p}\cTau^{1/2-\delta_p}.
 \label{eq:J-interpolation}
\end{align}
Both estimates are equalities for $p\in\{1,2,\infty\}$ and, for every
$p$, at $\tau=1$.
\end{corollary}

The exact values relevant to later comparisons are summarized below.
\begin{table}[ht]
\centering
\caption{Exact skew lower and upper power-mean radii.}
\label{tab:exact-values}
\renewcommand{\arraystretch}{1.25}
\begin{tabular}{@{}lll@{}}
\toprule
Setting & $\Ssk$ & $\Jsk$ \\
\midrule
$H$, $1\leq t\leq2$
  & $\M_t(1+\tau,|1-\tau|)$ & $\sqrt{1+\tau^2}$ \\
$H$, $2\leq t\leq\infty$
  & $\sqrt{1+\tau^2}$ & $\M_t(1+\tau,|1-\tau|)$ \\
$L_p$, outer exponent $p$
  & $(1+\tau^2)/\nu_p(\tau)$ & $\nu_p(\tau)$ \\
$L_p$, $\tau=1$
  & $2^{\min\{1/p,1-1/p\}}$ & $2^{\max\{1/p,1-1/p\}}$ \\
\bottomrule
\end{tabular}
\end{table}

Finally, \cref{prop:structure}(iii) transfers every $p$-anchor in
\cref{thm:Lp} to estimates for a different outer mean.  For example, if
$p\leq t\leq\infty$, then
\[
 \frac{\cTau}{\nu_p(\tau)}
 \leq\Ssk_{L_p,t}(\tau)
 \leq2^{1/p-1/t}\frac{\cTau}{\nu_p(\tau)}.
\]
For $1\leq t\leq p$, the inequalities are
\[
 2^{1/p-1/t}\frac{\cTau}{\nu_p(\tau)}
 \leq\Ssk_{L_p,t}(\tau)
 \leq\frac{\cTau}{\nu_p(\tau)}.
\]
These estimates retain the exact value at $t=p$ and avoid a separate
nonlinear optimization for every outer exponent.

\section{Sharp bounds and uniform nonsquareness}\label{sec:unsquare}

The next theorem identifies the correct universal scale.  Its lower bound
does not come from either reverse triangle inequality separately; it comes
from applying the skew transform twice.

\begin{theorem}\label{thm:sharp-chain}
For every $\tau\geq0$ and $1\leq t\leq\infty$,
\begin{equation}\label{eq:sharp-chain}
  \rTau
  \leq \Ssk_{X,t}(\tau)
  \leq \hTau
  \leq \Jsk_t[\tau,X]
  \leq 1+\tau.
\end{equation}
Every bound in \eqref{eq:sharp-chain} is sharp over the class of nonzero
real Banach spaces.
\end{theorem}

\begin{proof}
Fix $x,y\in\SX$ and write
\[
  u=x+\tau y,
  \qquad v=\tau x-y,
  \qquad a=\norm{u},\quad b=\norm{v}.
\]
The involution identity \eqref{eq:A-intro} reads
\begin{equation}\label{eq:inverse-pair}
  u+\tau v=\cTau x,
  \qquad
  \tau u-v=\cTau y.
\end{equation}
Consequently,
\[
  \cTau\leq a+\tau b,
  \qquad
  \cTau\leq\tau a+b.
\]
Adding these inequalities and using the monotonicity of power means gives
\[
  \M_t(a,b)\geq\M_1(a,b)=\frac{a+b}{2}
  \geq\frac{\cTau}{1+\tau}=\rTau.
\]
Taking the infimum proves the first inequality.

Choosing $y=x$ gives the pair of norms $1+\tau,\abs{1-\tau}$, and hence
$\Ssk_{X,t}(\tau)\leq\hTau\leq\Jsk_t[\tau,X]$.  Finally, both entries in
the defining power mean are at most $1+\tau$, which proves the last bound.

It remains to verify sharpness.  In a one-dimensional real space every
pair of unit vectors is $(x,x)$ or $(x,-x)$, so
$\Ssk_{X,t}(\tau)=\Jsk_t[\tau,X]=\hTau$; this makes both middle bounds
sharp.  For the lower bound, first suppose $0\leq\tau\leq1$ and take
$X=\ell_1^2$.  With its canonical basis $(e_1,e_2)$, put
\[
 x=\frac{e_1+\tau e_2}{1+\tau},
 \qquad
 y=\frac{\tau e_1-e_2}{1+\tau}.
\]
Then $x,y\in\SX$ and
\[
 x+\tau y=\rTau e_1,
 \qquad
 \tau x-y=\rTau e_2.
\]
Thus $\Ssk_{\ell_1^2,t}(\tau)=\rTau$.  The case $\tau>1$ follows from
\cref{prop:structure}(ii).  Finally, in $\ell_1^2$ the choice
$x=e_1,y=e_2$ makes both output norms equal to $1+\tau$, proving sharpness
of the rightmost bound.
\end{proof}

\begin{remark}\label{rem:involution-vs-classical}
For the classical pair $x+\tau y,x-\tau y$, independent scaling leads to
an even convex sum without further modification.  For the skew pair, that
property fails because the coefficient $\tau$ changes sides.  Identity
\eqref{eq:inverse-pair} is the replacement: it gives the sharp lower scale
$\rTau$, but later arguments requiring homogeneous normalization must first
be symmetrized.  This distinction is responsible for the main proof in
\cref{sec:normal}.
\end{remark}

\begin{corollary}\label{cor:t-infty-limit}
For every $\tau\geq0$,
\[
  \lim_{t\to\infty}\Ssk_{X,t}(\tau)=\Ssk_{X,\infty}(\tau),
  \qquad
  \lim_{t\to\infty}\Jsk_t[\tau,X]=\Jsk_\infty[\tau,X].
\]
\end{corollary}

\begin{proof}
The quantities increase with $t$, while
$\M_\infty(a,b)\leq2^{1/t}\M_t(a,b)$.  Apply infimum and supremum and let
$t\to\infty$.
\end{proof}

\subsection{The strict lower gap}

Recall that $X$ is uniformly nonsquare if there is $\eta>0$ such that for
all $x,y\in\SX$,
\begin{equation}\label{eq:UNS-definition}
 \min\{\norm{x+y},\norm{x-y}\}\leq2(1-\eta).
\end{equation}
Equivalently, its James constant
\[
 J(X):=\sup_{x,y\in\SX}
       \min\{\norm{x+y},\norm{x-y}\}
\]
is strictly smaller than $2$ \cite{James}.  The lower bound in
\cref{thm:sharp-chain} detects exactly this property.

\begin{theorem}\label{thm:gap-characterization}
Let $\tau>0$ and $1\leq t\leq\infty$, and exclude only the pair
$(t,\tau)=(1,1)$.  Then the following are equivalent:
\begin{enumerate}[label=\textup{(\roman*)}]
 \item $X$ is uniformly nonsquare;
 \item $\displaystyle
       \Ssk_{X,t}(\tau)>\frac{1+\tau^2}{1+\tau}$.
\end{enumerate}
Consequently, for any fixed $t>1$, strictness at one positive value of
$\tau$ is equivalent to strictness at every positive value of $\tau$.
\end{theorem}

\begin{proof}
By \cref{prop:structure}(ii), it suffices to consider $0<\tau\leq1$.  We
first prove the contrapositive of (i)$\Rightarrow$(ii).  Suppose that $X$
is not uniformly nonsquare.  There are $p_n,q_n\in\SX$ such that
\begin{equation}\label{eq:square-sequences}
 \norm{p_n+q_n}\longrightarrow2,
 \qquad
 \norm{p_n-q_n}\longrightarrow2.
\end{equation}
For $0<\tau\leq1$, the triangle inequality and \eqref{eq:square-sequences}
give
\begin{align*}
 \norm{p_n+\tau q_n}
 &\geq\norm{p_n+q_n}-(1-\tau)\longrightarrow1+\tau,\\
 \norm{\tau p_n-q_n}
 &\geq\norm{p_n-q_n}-(1-\tau)\longrightarrow1+\tau.
\end{align*}
The opposite inequalities are automatic, so both limits equal $1+\tau$.
Put
\[
  x_n=\frac{p_n+\tau q_n}{1+\tau},
  \qquad
  y_n=\frac{\tau p_n-q_n}{1+\tau}.
\]
Then $\norm{x_n},\norm{y_n}\to1$, while the involution identity gives
\begin{equation}\label{eq:almost-lower-pair}
 x_n+\tau y_n=\rTau p_n,
 \qquad
 \tau x_n-y_n=\rTau q_n.
\end{equation}
Normalize $x_n$ and $y_n$ separately.  Since the normalization errors tend
to zero, \eqref{eq:almost-lower-pair} produces unit pairs whose two output
norms converge to $\rTau$.  Therefore
$\Ssk_{X,t}(\tau)\leq\rTau$, and equality follows from
\cref{thm:sharp-chain}.

Conversely, suppose $\Ssk_{X,t}(\tau)=\rTau$.  Choose $x_n,y_n\in\SX$ and
set
\[
 u_n=x_n+\tau y_n,
 \quad v_n=\tau x_n-y_n,
 \quad a_n=\norm{u_n},
 \quad b_n=\norm{v_n},
\]
so that $\M_t(a_n,b_n)\to\rTau$.  The proof of
\cref{thm:sharp-chain} shows
\begin{equation}\label{eq:sum-convergence}
 \rTau\leq\frac{a_n+b_n}{2}
 \leq\M_t(a_n,b_n)\longrightarrow\rTau.
\end{equation}
It also gives the nonnegative errors
\begin{align*}
 e_n&:=a_n+\tau b_n-\cTau\geq0,\\
 d_n&:=\tau a_n+b_n-\cTau\geq0,
\end{align*}
whose sum is
\[
 e_n+d_n=(1+\tau)(a_n+b_n)-2\cTau\longrightarrow0.
\]
If $0<\tau<1$, subtracting the two relations shows
$(1-\tau)(a_n-b_n)=e_n-d_n\to0$.  Thus
$a_n,b_n\to\rTau$.  If $\tau=1$ and $1<t<\infty$, the same conclusion
follows from \eqref{eq:sum-convergence} and the strict convexity of the
scalar $\ell_t^2$ norm: the only nonnegative pair with arithmetic mean and
$t$-mean both equal to $1$ is $(1,1)$.  For $t=\infty$, it follows directly
from $\max\{a_n,b_n\}\to1$ and $a_n+b_n\geq2$.  This is the point at which
the excluded case $(t,\tau)=(1,1)$ would fail.

For large $n$, define $p_n=u_n/a_n$ and $q_n=v_n/b_n$.  Choose
$f_n,g_n\in S_{X^*}$ norming the two identities in
\eqref{eq:inverse-pair}, so that
\begin{align*}
 \cTau
 &=f_n(u_n+\tau v_n)
   =a_n f_n(p_n)+\tau b_n f_n(q_n),\\
 \cTau
 &=g_n(\tau u_n-v_n)
   =\tau a_n g_n(p_n)-b_n g_n(q_n).
\end{align*}
Since $e_n,d_n\to0$ and $a_n,b_n\to\rTau>0$, we obtain
\[
 f_n(p_n)\to1,
 \quad f_n(q_n)\to1,
 \qquad
 g_n(p_n)\to1,
 \quad g_n(q_n)\to-1.
\]
Hence
\[
 \norm{p_n+q_n}\geq f_n(p_n+q_n)\longrightarrow2,
 \qquad
 \norm{p_n-q_n}\geq g_n(p_n-q_n)\longrightarrow2.
\]
Thus $J(X)=2$, so $X$ is not uniformly nonsquare.  This proves the reverse
contrapositive and completes the equivalence.
\end{proof}

\begin{remark}\label{rem:exception}
When $t=1$ and $\tau=1$, every Banach space satisfies
\[
 \Ssk_{X,1}(1)\leq\M_1(\norm{2x},0)=1,
\]
and the universal lower bound is also $1$.  Thus no geometric information
can be recovered in this single degenerate case.  If $t=1$ but
$\tau\neq1$, the two distinct linear inequalities in the proof force
$a_n-b_n\to0$, so the characterization remains valid.
\end{remark}

\begin{corollary}\label{cor:gap-reflexivity}
If $(t,\tau)\neq(1,1)$, $\tau>0$, and
\[
 \Ssk_{X,t}(\tau)>\frac{1+\tau^2}{1+\tau},
\]
then $X$ is reflexive.
\end{corollary}

\begin{proof}
By \cref{thm:gap-characterization}, $X$ is uniformly nonsquare.  James's
theorem states that every uniformly nonsquare Banach space is reflexive
\cite{James}.
\end{proof}

\begin{example}\label{ex:gap-models}
For $1<p<\infty$, $L_p$ is uniformly nonsquare \cite{Clarkson}, and hence
\[
 \frac{\cTau}{\nu_p(\tau)}>\frac{\cTau}{1+\tau}
 \qquad(\tau>0).
\]
Equivalently, $\nu_p(\tau)<1+\tau$.  At $p=1$ and $p=\infty$, equality
holds.  Thus the abstract gap theorem recovers, through
\cref{thm:Lp}, the strict separation between the interior and the endpoints
of the Lebesgue scale.
\end{example}

\section{James-type estimates and uniform normal structure}\label{sec:normal}

For $0\leq\varepsilon\leq1$, define the classical James modulus
\begin{equation}\label{eq:James-modulus}
 J(\varepsilon,X)
 :=\sup_{x,y\in\SX}
   \min\{\norm{x+\varepsilon y},\norm{x-\varepsilon y}\}.
\end{equation}
Theorem~3.2 of Zuo and Tang \cite{ZuoTang} relates the usual power-mean
Sch\"affer constant to \eqref{eq:James-modulus}.  A direct repetition does
not work for the skew constant: the function
\[
 (\alpha,\beta)\longmapsto
 \norm{\alpha x+\tau\beta y}^t
 +\norm{\tau\alpha x-\beta y}^t
\]
need not be even in either variable.  The next lemma restores evenness by
adding the reversed orientation.

For $z,w\in X$, $\tau\geq0$, and $1\leq t<\infty$, put
\begin{align}
 \Gamma_{\tau,t}(z,w)
 :=&\ \norm{z+\tau w}^t+\norm{z-\tau w}^t\notag\\
   &+\norm{w+\tau z}^t+\norm{w-\tau z}^t.
   \label{eq:Gamma}
\end{align}

\begin{lemma}\label{lem:lifting}
Let $z,w\in X\setminus\{0\}$ and
$m=\min\{\norm{z},\norm{w}\}$.  Then
\begin{equation}\label{eq:lifting}
 \Gamma_{\tau,t}(z,w)
 \geq4m^t\Ssk_{X,t}(\tau)^t.
\end{equation}
\end{lemma}

\begin{proof}
Write $z=\alpha p,w=\beta q$ with $p,q\in\SX$ and
$\alpha,\beta>0$.  For fixed $\beta$, each of
\begin{align*}
 s&\longmapsto\norm{sp+\tau\beta q}^t
                  +\norm{sp-\tau\beta q}^t,\\
 s&\longmapsto\norm{\beta q+\tau sp}^t
                  +\norm{\beta q-\tau sp}^t
\end{align*}
is even and convex.  It is therefore nondecreasing on $[0,\infty)$.  The
same statement holds when the roles of the two scaling variables are
interchanged.  Since $\alpha,\beta\geq m$,
\[
 \Gamma_{\tau,t}(\alpha p,\beta q)
 \geq\Gamma_{\tau,t}(mp,mq)
 =m^t\Gamma_{\tau,t}(p,q).
\]
For unit $p,q$, the four terms split into the two skew orientations:
\begin{align*}
 \Gamma_{\tau,t}(p,q)
  =&\ \bigl(\norm{p+\tau q}^t+\norm{\tau p-q}^t\bigr)\\
   &+\bigl(\norm{q+\tau p}^t+\norm{\tau q-p}^t\bigr).
\end{align*}
Each parenthesis is at least $2\Ssk_{X,t}(\tau)^t$, proving
\eqref{eq:lifting}.
\end{proof}

\begin{theorem}\label{thm:James-estimate}
Let $0\leq\tau,\varepsilon\leq1$ and $1\leq t<\infty$.  Set
\begin{equation}\label{eq:AB}
 A_{\tau,\varepsilon}
 :=1+\tau+\varepsilon(1-\tau),
 \qquad
 B_{\tau,\varepsilon}
 :=1-\tau+\varepsilon(1+\tau).
\end{equation}
Then
\begin{equation}\label{eq:James-main}
 2\Ssk_{X,t}(\tau)^tJ(\varepsilon,X)^t
 \leq A_{\tau,\varepsilon}^t+B_{\tau,\varepsilon}^t.
\end{equation}
Equivalently,
\begin{equation}\label{eq:James-mean}
 \Ssk_{X,t}(\tau)J(\varepsilon,X)
 \leq\M_t(A_{\tau,\varepsilon},B_{\tau,\varepsilon}).
\end{equation}
The constant in \eqref{eq:James-main} is sharp: equality holds in
$X=\ell_1^2$ when $\tau=\varepsilon=1$.
\end{theorem}

\begin{proof}
Fix $x,y\in\SX$ and apply \cref{lem:lifting} to
\[
  z=x+\varepsilon y,
  \qquad
  w=x-\varepsilon y.
\]
If $z=0$ or $w=0$, the pointwise estimate below is immediate because its
left-hand side contains $\min\{\norm{z},\norm{w}\}^t=0$.  We may therefore
assume that both vectors are nonzero and invoke the lemma.
The four vectors in \eqref{eq:Gamma} are
\begin{align*}
 z+\tau w&=(1+\tau)x+\varepsilon(1-\tau)y,\\
 z-\tau w&=(1-\tau)x+\varepsilon(1+\tau)y,\\
 w+\tau z&=(1+\tau)x-\varepsilon(1-\tau)y,\\
 w-\tau z&=(1-\tau)x-\varepsilon(1+\tau)y.
\end{align*}
Their norms are bounded, in order, by
$A_{\tau,\varepsilon},B_{\tau,\varepsilon},
 A_{\tau,\varepsilon},B_{\tau,\varepsilon}$.  Consequently,
\begin{align*}
 4\Ssk_{X,t}(\tau)^t
 \min\{\norm{x+\varepsilon y},\norm{x-\varepsilon y}\}^t
 &\leq\Gamma_{\tau,t}(z,w)\\
 &\leq2A_{\tau,\varepsilon}^t
       +2B_{\tau,\varepsilon}^t.
\end{align*}
Take the supremum over $x,y\in\SX$ to obtain
\eqref{eq:James-main}.

For sharpness, $\Ssk_{\ell_1^2,t}(1)=1$ by
\cref{thm:sharp-chain}, while $J(1,\ell_1^2)=2$.  At
$\tau=\varepsilon=1$, both quantities in \eqref{eq:AB} equal $2$, so the
two sides of \eqref{eq:James-main} are equal to $2^{t+1}$.
\end{proof}

\begin{corollary}\label{cor:diagonal}
For $0\leq\tau\leq1$ and $1\leq t<\infty$,
\begin{equation}\label{eq:diagonal}
 2\Ssk_{X,t}(\tau)^tJ(\tau,X)^t
 \leq(1+2\tau-\tau^2)^t+(1+\tau^2)^t.
\end{equation}
For $t=\infty$, the limiting forms of
\eqref{eq:James-mean} and \eqref{eq:diagonal} are, respectively,
\begin{align*}
 \Ssk_{X,\infty}(\tau)J(\varepsilon,X)
 &\leq A_{\tau,\varepsilon},\\
 \Ssk_{X,\infty}(\tau)J(\tau,X)
 &\leq1+2\tau-\tau^2.
\end{align*}
\end{corollary}

\begin{proof}
Set $\varepsilon=\tau$ in \cref{thm:James-estimate}.  The limiting
statements follow from \cref{cor:t-infty-limit} and the convergence of
power means.  Notice that
$A_{\tau,\varepsilon}\geq B_{\tau,\varepsilon}$ on the stated parameter
range.
\end{proof}

We finish with the normal-structure consequence.  A Banach space $X$ has
\emph{uniform normal structure} if there is $0<c<1$ such that every
non-singleton bounded closed convex set $C\subset X$ contains a point
$x_0$ satisfying
\[
 \sup_{y\in C}\norm{x_0-y}
 \leq c\,\sup_{x,y\in C}\norm{x-y}.
\]
This property is central in metric fixed point theory; see
\cite{GoebelKirk}.  Saejung proved \cite{Saejung} that $X$ has uniform
normal structure whenever
\begin{equation}\label{eq:Saejung}
 J(\varepsilon,X)<g(\varepsilon)
 :=\frac{\varepsilon+\sqrt{4+\varepsilon^2}}{2}
\end{equation}
for some $0<\varepsilon\leq1$.

\begin{corollary}
\label{cor:normal-criterion}
Let $0\leq\tau\leq1$, $0<\varepsilon\leq1$, and
$1\leq t\leq\infty$.  If
\begin{equation}\label{eq:normal-general}
 \Ssk_{X,t}(\tau)
 >\frac{\M_t(A_{\tau,\varepsilon},B_{\tau,\varepsilon})}
         {g(\varepsilon)},
\end{equation}
then $X$ has uniform normal structure.  In particular, it is enough that
for some $0<\tau\leq1$,
\begin{equation}\label{eq:normal-diagonal}
 \Ssk_{X,t}(\tau)
 >\frac{1}{g(\tau)}
  \left(\frac{(1+2\tau-\tau^2)^t+(1+\tau^2)^t}{2}\right)^{1/t}
\end{equation}
when $t<\infty$, with the maximum replacing the power mean for
$t=\infty$.
\end{corollary}

\begin{proof}
Under \eqref{eq:normal-general}, \cref{thm:James-estimate} gives
$J(\varepsilon,X)<g(\varepsilon)$.  Apply Saejung's criterion
\eqref{eq:Saejung}.  The diagonal statement follows by taking
$\varepsilon=\tau$.
\end{proof}

\begin{remark}
\label{rem:Gao-threshold}
At $\tau=\varepsilon=1$ and $t=2$, condition
\eqref{eq:normal-general} becomes
\[
 \Ssk_{X,2}(1)>\frac{2}{(1+\sqrt5)/2}.
\]
If Gao's quadratic parameter is written as
$f(X)=2\Ssk_{X,2}(1)^2$, this is exactly
\[
 f(X)>12-4\sqrt5,
\]
the familiar Pythagorean threshold discussed in
\cite{Gao,Saejung,ZuoTang}.  The two-parameter form
\eqref{eq:normal-general} is stronger as a certificate: one may keep the
skew parameter $\tau$ independent of the James scale $\varepsilon$ and
optimize the resulting bound for the space under consideration.
\end{remark}

\begin{remark}
The supremal invariant \eqref{eq:skew-james-intro} and the infimal invariant
\eqref{eq:def-S-intro} coincide in Hilbert space at the quadratic exponent,
but they encode opposite operator radii on $L_p$.  The product identity
\eqref{eq:product-identity}, the gap characterization
\cref{thm:gap-characterization}, and the symmetrized estimate
\cref{thm:James-estimate} show that the infimum theory is not a formal
restatement of the skew James theory.  Rather, the two theories are linked
by the scaled involution $A_\tau^2=(1+\tau^2)I$, which is the common
algebraic source of their exact model behavior.
\end{remark}


\begin{thebibliography}{99}
\footnotesize
\setlength{\itemsep}{0pt}

\bibitem{BerghLofstrom}
J.~Bergh and J.~L\"ofstr\"om,
\emph{Interpolation Spaces: An Introduction},
Grundlehren der mathematischen Wissenschaften, vol.~223,
Springer, Berlin, 1976.

\bibitem{Day}
M.~M.~Day,
Some characterizations of inner-product spaces,
\emph{Trans. Amer. Math. Soc.} \textbf{62} (1947), 320--337.

\bibitem{Clarkson}
J.~A.~Clarkson,
Uniformly convex spaces,
\emph{Trans. Amer. Math. Soc.} \textbf{40} (1936), 396--414.

\bibitem{Gao}
J.~Gao,
A Pythagorean approach in Banach spaces,
\emph{J. Inequal. Appl.} (2006), Art.~ID 94982, 11 pp.,
\href{https://doi.org/10.1155/JIA/2006/94982}{doi:10.1155/JIA/2006/94982}.

\bibitem{GoebelKirk}
K.~Goebel and W.~A.~Kirk,
\emph{Topics in Metric Fixed Point Theory},
Cambridge Studies in Advanced Mathematics, vol.~28,
Cambridge University Press, Cambridge, 1990.

\bibitem{James}
R.~C.~James,
Uniformly non-square Banach spaces,
\emph{Ann. of Math. (2)} \textbf{80} (1964), 542--550,
\href{https://doi.org/10.2307/1970663}{doi:10.2307/1970663}.

\bibitem{JordanVonNeumann}
P.~Jordan and J.~von Neumann,
On inner products in linear, metric spaces,
\emph{Ann. of Math. (2)} \textbf{36} (1935), 719--723,
\href{https://doi.org/10.2307/1968653}{doi:10.2307/1968653}.

\bibitem{KatoMaligrandaTakahashi}
M.~Kato, L.~Maligranda, and Y.~Takahashi,
On James and Jordan--von Neumann constants and the normal structure
coefficient of Banach spaces,
\emph{Studia Math.} \textbf{144} (2001), 275--295,
\href{https://doi.org/10.4064/sm144-3-5}{doi:10.4064/sm144-3-5}.

\bibitem{LiuLi}
Q.~Liu and Y.~Li,
On a new geometric constant related to the Euler--Lagrange type identity in
Banach spaces,
\emph{Mathematics} \textbf{9} (2021), Art.~116,
\href{https://doi.org/10.3390/math9020116}{doi:10.3390/math9020116}.

\bibitem{LiuZhouSarfrazLi}
Q.~Liu, C.~Zhou, M.~Sarfraz, and Y.~Li,
On new moduli related to the generalization of the parallelogram law,
\emph{Bull. Malays. Math. Sci. Soc.} \textbf{45} (2022), 307--321,
\href{https://doi.org/10.1007/s40840-021-01196-7}{doi:10.1007/s40840-021-01196-7}.

\bibitem{MartiniPapiniWu}
H.~Martini, P.~L.~Papini, and S.~Wu,
Variations of the James and Sch\"affer constants in Banach spaces,
\emph{Ann. Funct. Anal.} \textbf{15} (2024), Art.~7,
\href{https://doi.org/10.1007/s43034-023-00308-7}{doi:10.1007/s43034-023-00308-7}.

\bibitem{NiLiuZhou}
Q.~Ni, Q.~Liu, and Y.~Zhou,
Skew generalized von Neumann--Jordan constant in Banach spaces,
\emph{Hacet. J. Math. Stat.} \textbf{54} (2025), 436--444,
\href{https://doi.org/10.15672/hujms.1451136}{doi:10.15672/hujms.1451136}.

\bibitem{RaoWuLiuNiWang}
Z.-Y.~Rao, Q.~Wu, Q.~Liu, Q.-C.~Ni, and Y.-X.~Wang,
The skew James type constant in Banach spaces,
\emph{Anal. Math.} \textbf{52} (2026), 619--636,
\href{https://doi.org/10.1007/s10476-026-00166-0}{doi:10.1007/s10476-026-00166-0}.

\bibitem{Saejung}
S.~Saejung,
Sufficient conditions for uniform normal structure of Banach spaces and
their duals,
\emph{J. Math. Anal. Appl.} \textbf{330} (2007), 597--604,
\href{https://doi.org/10.1016/j.jmaa.2006.07.087}{doi:10.1016/j.jmaa.2006.07.087}.

\bibitem{Schaffer}
J.~J.~Sch\"affer,
\emph{Geometry of Spheres in Normed Spaces},
Lecture Notes in Pure and Applied Mathematics, vol.~20,
Marcel Dekker, New York, 1976.

\bibitem{Takahashi}
Y.~Takahashi,
Some geometric constants of Banach spaces: a unified approach,
in \emph{Banach and Function Spaces II} (M.~Kato and L.~Maligranda, eds.),
Yokohama Publishers, Yokohama, 2007, pp.~191--220.

\bibitem{WangYang}
F.~Wang and C.~Yang,
An inequality between the James and James type constants in Banach spaces,
\emph{Studia Math.} \textbf{201} (2010), 191--201,
\href{https://doi.org/10.4064/sm201-2-5}{doi:10.4064/sm201-2-5}.

\bibitem{YangYangBF}
{\
C.~Yang and X.~Yang,
On the James type constant and von Neumann--Jordan constant for a class of
Bana\'s--Fr\k{a}czek type spaces,
\emph{J. Math. Inequal.} \textbf{10} (2016), 551--558,
\href{https://doi.org/10.7153/jmi-10-43}{doi:10.7153/jmi-10-43}.}

\bibitem{ZuoTang}
Z.-F.~Zuo and C.-L.~Tang,
Sch\"affer-type constant and uniform normal structure in Banach spaces,
\emph{Ann. Funct. Anal.} \textbf{7} (2016), 452--461,
\href{https://doi.org/10.1215/20088752-3605636}{doi:10.1215/20088752-3605636}.

\end{thebibliography}
\end{document}